\documentclass{amsart}

\usepackage[cp1251]{inputenc}
\usepackage[T2A]{fontenc}
\usepackage{amsmath}
\usepackage{amsfonts}
\usepackage{amssymb}
\usepackage{amsthm}
\usepackage{euscript}
\usepackage[english]{babel}
\theoremstyle{plain}
\newtheorem{theorem}{Theorem}

\begin{document}
\title[Spectral gaps of the Hill--Schr\"{o}dinger operators]{Spectral gaps of the Hill--Schr\"{o}dinger operators with 
distributional potentials}

%----------------------------------- Author 1 ------------------------------
\author[V. Mikhailets]{Vladimir Mikhailets}

\address{Institute of Mathematics \\
         National Academy of Science of Ukraine \\
         3~Tereshchekivs'ka str. \\
         01601 Kyiv-4 \\
         Ukraine}

\address{National Technical University of Ukraine \\
         "Kyiv Politecknic Institute" \\
         37~Peremogy~Av. \\
         03056 Kyiv \\
         Ukraine}         
         
\email{mikhailets@imath.kiev.ua}

%----------------------------------- Author 2 ------------------------------
\author[V. Molyboga]{Volodymyr Molyboga}

\address{Institute of Mathematics \\
         National Academy of Science of Ukraine \\
         3~Tereshchekivs'ka str. \\
         01601 Kyiv-4 \\
         Ukraine}

\email{molyboga@imath.kiev.ua}

%%%%%%%%%%%%%%%%%%%%%%%%%%%%%%%%%%%%%%%%%%%%%%%%%%%%%%%%%%%%%%%%%%%%%%%%%%%%%%%%%%%%%%%%%%%%%%%%%%%%%%%%%%%%%%%%%%%%%%%%%%%%
%%%%%%%%%%%%%%%%%%%%%%%%%%%%%%%%%%%%%%%%%%%%%%%%%%%%%%%%%%%%%%%%%%%%%%%%%%%%%%%%%%%%%%%%%%%%%%%%%%%%%%%%%%%%%%%%%%%%%%%%%%%%



\dedicatory{Dedicated to Professor V. M. Adamyan on the occasion of his 75 birthday}

%-------------------------------------------------------------------------------
% 34-XX | ORDINARY DIFFERENTIAL EQUATIONS
% 34Lxx | Ordinary differential operators
% 34L40 | Particular operators (Dirac, one-dimensional Schr\"{o}dinger, etc.)
%-------------------------------------------------------------------------------
% 47-XX | OPERATOR THEORY
% 47Axx | General theory of linear operators
% 47A10 | Spectrum, resolvent
% 47A75 | Eigenvalue problems [See also 47J10, 49R05]
%-------------------------------------------------------------------------------

%%%%%%%%%%%%%%%%%%%%%%%%%%%%%%%%%%%%%%%%%%%%%%%%%%%%%%%%%%%%%%%%%%%%%%%%%%%%%%%%%%%%%%%%%%%%%%%%%%%%%%%%%%%%%%%%%%%%%%%%%%%%
%%%%%%%%%%%%%%%%%%%%%%%%%%%%%%%%%%%%%%%%%%%%%%%%%%%%%%%%%%%%%%%%%%%%%%%%%%%%%%%%%%%%%%%%%%%%%%%%%%%%%%%%%%%%%%%%%%%%%%%%%%%%
\begin{abstract}
The paper studies the Hill--Schr\"{o}dinger operators with potentials in the space $H^\omega \subset H^{-1}\left(\mathbb{T}, 
\mathbb{R}\right)$. 
The main results completely describe the sequences arising as the lengths of spectral gaps of these 
operators. The space $H^\omega$ coincides with the H\"{o}rmander space $H^{\omega}_2\left(\mathbb{T}, \mathbb{R}\right)$ with 
the weight function $\omega(\sqrt{1+\xi^{2}})$ if $\omega$ belongs to Avakumovich's class $\mathrm{OR}$. 
In particular, if the functions $\omega$ are power, then these spaces coincide with the Sobolev spaces. 
The functions $\omega$ may be nonmonotonic. 
\end{abstract}

\maketitle
%%%%%%%%%%%%%%%%%%%%%%%%%%%%%%%%%%%%%%%%%%%%%%%%%%%%%%%%%%%%%%%%%%%%%%%%%%%%%%%%%%%%%%%%%%%%%%%%%%%%%%%%%%%%%%%%%%%%%%%%%%%%
%%%%%%%%%%%%%%%%%%%%%%%%%%%%%%%% !!! Text of the paper !!! %%%%%%%%%%%%%%%%%%%%%%%%%%%%%%%%%%%%%%%%%%%%%%%%%%%%%%%%%%%%%%%%%
\section{Introduction}
Let us consider the Hill--Schr\"{o}dinger operator
\begin{equation}\label{eq_10}
  \mathrm{S}(q)u:=-u''+q(x)u,\quad x\in \mathbb{R},
\end{equation}
with 1-periodic real-valued potential
\begin{equation*}
  q(x)=\sum_{k\in \mathbb{Z}}\widehat{q}(k)e^{i k 2\pi x}\in L^{2}(\mathbb{T},\mathbb{R}),\quad
  \mathbb{T}:=\mathbb{R}/\mathbb{Z}.
\end{equation*}
This condition means that
\begin{equation*}\label{eq_11}
 \sum_{k\in \mathbb{Z}}|\widehat{q}(k)|^{2}<\infty
 \quad\text{and}\quad \widehat{q}(k)=\overline{\widehat{q}(-k)},\quad k\in \mathbb{Z}.
\end{equation*}

It is well known that the operator $\mathrm{S}(q)$ is lower semibounded and self-adjoint in the Hilbert space 
$L^{2}(\mathbb{R})$. Its spectrum is absolutely continuous and has a zone structure~\cite{ReSi4}.

Spectrum of the operator $\mathrm{S}(q)$ is completely defined by the location of the endpoints of spectral gaps 
$\{\lambda_{0}(q),\lambda_{n}^{\pm}(q)\}_{n=1}^{\infty}$, which satisfy the inequalities
\begin{equation}\label{InEq}
  -\infty<\lambda_{0}(q)<\lambda_{1}^{-}(q)\leq\lambda_{1}^{+}(q)<\lambda_{2}^{-}(q)\leq\lambda_{2}^{+}(q)<\cdots\,.
\end{equation}
Some gaps may be degenerate, then the corresponding bands merge. 
For even/odd numbers $n\in \mathbb{Z}_{+}$, the endpoints of spectral gaps 
$\{\lambda_{0}(q),\lambda_{n}^{\pm}(q)\}_{n=1}^{\infty}$ 
are eigenvalues of the periodic/semi\-periodic problems on the interval $(0,1)$.

The interiors of spectral bands (the stability zones)
\begin{equation*}
  \mathcal{B}_{0}(q):=(\lambda_{0}(q),\lambda_{1}^{-}(q)),\qquad
  \mathcal{B}_{n}(q):=(\lambda_{n}^{+}(q),\lambda_{n+1}^{-}(q)),\quad n\in
  \mathbb{N},
\end{equation*}
together with the collapsed gaps,
\begin{equation*}
 \lambda=\lambda_{n_{i}}^{-}=\lambda_{n_{i}}^{+},
\end{equation*}
are characterized as the set of those $\lambda\in \mathbb{R}$, for which all solutions of the equation
\begin{equation}\label{MnEq}
  -u''+q(x)u=\lambda u, \qquad x\in \mathbb{R},
\end{equation}
are bounded on $\mathbb{R}$. 
The open spectral gaps (instability zones)
\begin{equation*}
  \mathcal{G}_{0}(q):=(-\infty,\lambda_{0}(q)),\qquad
  \mathcal{G}_{n}(q):=(\lambda_{n}^{-}(q),\lambda_{n}^{+}(q))\neq\emptyset,\quad n\in
  \mathbb{N},
\end{equation*}
form a set of those $\lambda\in \mathbb{R}$ for which any nontrivial solution of the equation \eqref{MnEq} 
is unbounded on $\mathbb{R}$.

We study the behaviour of the lengths of spectral gaps,
\begin{equation*}
  \gamma_{q}(n):=\lambda_{n}^{+}(q)-\lambda_{n}^{-}(q),\quad n\in \mathbb{N},
\end{equation*}
of the operator $\mathrm{S}(q)$ in terms of behaviour of the Fourier coefficients $\{\widehat{q}(n)\}_{n\in \mathbb{N}}$ of 
the potential $q$ with respect to test sequence spaces, that is in terms of potential regularity.

Hochstadt \cite{Hchs1, Hchs2}, Marchenko and Ostrovskii \cite{MrOs}, McKean and Trubowitz \cite{McKTr, Trb} proved that the 
potential $q$ is an infinitely differentiable function if and only if the lengths of spectral gaps 
$\{\gamma_{q}(n)\}_{n=1}^{\infty}$ decrease faster than an arbitrary power of~$1/n$:
\begin{equation*}
  q\in C^{\infty}(\mathbb{T})\Leftrightarrow
  \gamma_{q}(n)=O(n^{-k}),\; n\rightarrow\infty,\quad k\in \mathbb{Z}_{+}.
\end{equation*}

However, the scale of spaces $\left\{C^{k}(\mathbb{T})\right\}_{k \in \mathbb{N}}$ turned out unsuitable to 
obtain precise quantitative results. Marchenko and Ostrovskii \cite{MrOs} (see also \cite{Mrch, Lvt1987}) found that
\begin{equation}\label{eq_14}
  q\in H^{s}(\mathbb{T})\Leftrightarrow \{\gamma_{q}(n)\}_{n\in \mathbb{N}}\in h^{s}(\mathbb{N}),\qquad s\in  \mathbb{Z}_{+}.
\end{equation}

The Sobolev spaces $H^{s}(\mathbb{T})$, $s\in \mathbb{R}$, of 1-periodic functions/generalized functions 
may also be defined by means of their Fourier coefficients
\begin{equation*}\label{eq_16}
   H^{s}(\mathbb{T})  =\left\{f=\sum_{k\in \mathbb{Z}}\widehat{f}\,(k)e^{i k2\pi x}\in 
\mathfrak{D}\left(\mathbb{T}\right)\left|\;  \lVert f\rVert_{H^{s}(\mathbb{T})}^{2}:=\sum_{k\in 
\mathbb{Z}}(1+|k|)^{2s}|\widehat{f}(k)|^{2}<\infty\right.\right\}.
\end{equation*}
Here by $\mathfrak{D}(\mathbb{T})$ we denote the space of 1-periodic generalized functions on~$\mathbb{T}$.

We define the weighted sequence spaces $h^{s}(\mathbb{N})$, $s\in \mathbb{R}$, 
in the following way:
\begin{equation*}\label{eq_18}
  h^{s}(\mathbb{N}):= \left\{a=\{a(k)\}_{k\in \mathbb{N}}\left|\; \lVert a\rVert_{h^{s}(\mathbb{N})}^{2}:=\sum_{k\in 
\mathbb{N}}(1+|k|)^{2s}|a(k)|^{2}<\infty\right.\right\}.
\end{equation*}

Marchenko--Ostrovskii theorem \eqref{eq_14} can be extended to a more general scale of H\"{o}r\-man\-der spaces 
$\{H^{\omega}(\mathbb{T})\}_{\omega}$ \cite{DjMt2, DjMt3, PschMA2011}, 
where $\omega=\{\omega(k)\}_{k\in \mathbb{Z}}$ is a weighted sequence. 
Recall that a sequence  $a=\{a(k)\}_{k\in \mathbb{Z}}$ is called a weight or weighted sequence 
if it is positive and even, i. e., $a(k)\geq 0$ и $a(-k)=a(k)$ for $k\in \mathbb{Z}_{+}$.

However, complete description of the sequences that form lengths of the gaps with potentials from the given functional 
class, in particular the H\"{o}rmander space or the Sobolev space, remained an open question. This paper deals with this 
issue in more general situation of distributional potentials.

%%%%%%%%%%%%%%%%%%%%%%%%%%%%%%%%%%%%%%%%%%%%%%%%%%%%%%%%%%%%%%%%%%%%%%%%%%%%%%%%%%%%%%%%%%%%%%%%%%%%%%%%%%%%%%%%%%%%%%%%%%%%
%%%%%%%%%%%%%%%%%%%%%%%%%%%%%%%%%%%%%%%%%%%%%%%%%%%%%%%%%%%%%%%%%%%%%%%%%%%%%%%%%%%%%%%%%%%%%%%%%%%%%%%%%%%%%%%%%%%%%%%%%%%%
\section{Main results}
Let us start with necessary notations. 
The spaces $H^{\omega}(\mathbb{T})$ and $h^{\omega}(\mathbb{N})$ are defined similarly to the spaces 
$H^{s}(\mathbb{T})$ and $h^{s}(\mathbb{N})$:
\begin{align*}\label{eq_20}
   H^{\omega}(\mathbb{T}) & :=\left\{f=\sum_{k\in \mathbb{Z}}\widehat{f}\,(k)e^{i k2\pi x}\in 
   \mathfrak{D}\left(\mathbb{T}\right)\left|\; \lVert f\rVert_{H^{\omega}(\mathbb{T})}^{2} :=\sum_{k\in 
   \mathbb{Z}}\omega^{2}(k)|\widehat{f}(k)|^{2}<\infty\right.\right\}, \\
    h^{\omega}(\mathbb{N}) & := \left\{a=\{a(k)\}_{k\in \mathbb{N}}\left|\; \lVert 
a\rVert_{h^{\omega}(\mathbb{N})}^{2}:=\sum_{k\in \mathbb{N}}\omega^{2}(k)|a(k)|^{2}<\infty\right.\right\}.
\end{align*}

We say that the weighted sequence $\omega=\{\omega(k)\}_{k\in \mathbb{Z}}$ 
belongs to the class $\mathrm{I}_{0}$, 
if it satisfies the following condition:
\begin{equation*}\label{eq_22}
 |k|^{s}\ll \omega(k)\ll |k|^{1+s},\qquad s\in [0,\infty).
\end{equation*}
The notation
\begin{equation*}\label{eq_24}
  b(k)\ll a(k)\ll c(k),\qquad k\in \mathbb{N},
\end{equation*}
means that there are positive constants $C_{1}$ и $C_{2}$ such that the following inequalities hold:
\begin{equation*}\label{eq_26}
  C_{1}b(k)\leq a(k)\leq C_{2}c(k),\qquad k\in \mathbb{N}.
\end{equation*}
We say that the weighted sequence $\omega=\{\omega(k)\}_{k\in \mathbb{Z}}$ 
belongs to the class  $\mathrm{M}_{0}$, 
if it satisfies the following conditions:
\begin{align*}\label{eq_28}
 \mathtt{(i)}\hspace{5pt} & \omega(k)\uparrow\infty,\; k\in \mathbb{N};\hspace{5pt}\text{(monotonicity)}  \hspace{200pt} \\
\mathtt{(ii)}\hspace{5pt} & \omega(k+m)\leq \omega(k)\omega(m),\quad k,m\in \mathbb{N};\hspace{5pt}\text{(submultiplicity)} \\
\mathtt{(iii)}\hspace{5pt} & \frac{\log\omega(k)}{k}\downarrow 0,\quad 
k\rightarrow\infty,\hspace{5pt}\text{(subexponentiality)}.
\end{align*}

Suppose that a weighted sequence $\omega=\{\omega(k)\}_{k\in \mathbb{Z}}$ belongs either to class 
$\mathrm{I}_{0}$ or to the class $\mathrm{M}_{0}$. Then 
\begin{equation}\label{eq_30}
 q\in H^{\omega}(\mathbb{T})\Leftrightarrow \{\gamma_{q}(n)\}_{n\in \mathbb{N}}\in  h^{\omega}(\mathbb{N}).
\end{equation} 
The statement \eqref{eq_30} for the case $\omega\in \mathrm{I}_{0}$ 
was proved by the authors \cite{MiMlOTAA2012}, 
and the case $\mathrm{M}_{0}$ was closely studied in \cite{DjMt2, PschMA2011}.

The statement \eqref{eq_30} may be strengthened. 
It is well-known that the sequence of lengths of spectral gaps 
$\{\gamma_{q}(n)\}_{n\in \mathbb{N}}$ of the Hill--Schr\"{o}dinger operator $\mathrm{S}(q)$ 
with an $L^{2}(\mathbb{T})$-potential $q$ 
belongs to the space 
\begin{equation*}\label{eq_32}
  h_{+}^{0}(\mathbb{N}):= 
   \left\{a=\{a(k)\}_{k\in \mathbb{N}} \in l^{2}(\mathbb{N})\left|\; a(k)\geq 0,\; k\in 
\mathbb{N}\right.\right\}.
\end{equation*} 
Let us consider the map
\begin{equation*}\label{eq_34}
  \gamma: L^{2}(\mathbb{T})\ni q\mapsto \{\gamma_{q}(n)\}_{n\in \mathbb{N}} \in h_{+}^{0}(\mathbb{N}).
\end{equation*}
Then due to Garnett and Trubowitz \cite{GrTr1984, GrTr1987},
\begin{equation*}
 \gamma\left(L^{2}(\mathbb{T})\right) = h_{+}^{0}(\mathbb{N}).
\end{equation*}

We introduce following notations:
\begin{equation*}
  h_{+}^{\omega}(\mathbb{N}):= \left\{a=\{a(k)\}_{k\in \mathbb{N}}\in h^{\omega}(\mathbb{N})\left|\; a(k)\geq 0,\; 
k\in\mathbb{N}\right.\right\}.
\end{equation*}
\begin{theorem}\label{th_10}
Suppose that $q\in L^{2}(\mathbb{T})$ and that either $\omega\in\mathrm{I}_{0}$ or 
$\omega\in \mathrm{M}_{0}$. 
Then the map
\begin{equation*}
  \gamma: L^{2}(\mathbb{T})\ni q\mapsto \{\gamma_{q}(n)\}_{n\in \mathbb{N}} \in h_{+}^{0}(\mathbb{N})
\end{equation*}
satisfies the relations
\begin{align*}
 \mathtt{(i)}&\quad \gamma\left(H^{\omega}(\mathbb{T})\right) = h_{+}^{\omega}(\mathbb{N}),\\ %\hspace{25pt}
 \mathtt{(ii)}&\quad \gamma^{-1}\left(h_{+}^{\omega}(\mathbb{N})\right) = H^{\omega}(\mathbb{T}). %\hspace{110pt}
 \end{align*}
\end{theorem}

%%%%%%%%%%%%%%%%%%%%%%%%%%%%%%%%%%%%%%%%%%%%%%%%%%%%%%%%%%%%%%%%%%%%%%%%%%%%%%%%%%%%%%%%%%%%%%%%%%%%%%%%%%%%%%%%%%%%%%%%%%%%
%%%%%%%%%%%%%%%%%%%%%%%%%%%%%%%%%%%%%%%%%%%%%%%%%%%%%%%%%%%%%%%%%%%%%%%%%%%%%%%%%%%%%%%%%%%%%%%%%%%%%%%%%%%%%%%%%%%%%%%%%%%%

%After the celebrated Kronig and Penney paper \cite{KrPn1931} the Schr\"{o}dinger operators with (periodic) distributions as 
%potentials came into mathematical physics. The subsequent development of quantum mechanics stimulated active growth of this 
%branch of science (see the bibliography of the monograph \cite{AlGHH2004, AlKr2000}).

Now let us consider the Hill--Schr\"{o}dinger operator $\mathrm{S}(q)$ 
with a 1-periodic real-valued distribution potential $q$ that belongs to the negative Sobolev space:
\begin{equation}\label{eq_MCP}
  q=\sum_{k\in \mathbb{Z}}\widehat{q}(k)e^{i k 2\pi x}\in H^{-1}(\mathbb{T}).
\end{equation}
All real-valued pseudo-functions, measures, pseudo-measures and some more singular distributions on the circle 
satisfy this condition. 
For more detailed discussion of operators with strongly singular potentials see \cite{EcGsNcTs2013} and 
references therein. 

Under the assumption \eqref{eq_MCP} the operator \eqref{eq_10} may be well defined in the complex Hilbert space 
$L^{2}(\mathbb{R})$ in the following basic ways:
\begin{itemize}
  \item as form-sum operator;
  \item as quasi-differential operators;
  \item as limit of operators with smooth 1-periodic potentials in the norm resolvent sense.
\end{itemize}
Equivalence of all these definitions was proved in the paper \cite{MiMl6}, 
more general case was treated in \cite{MiMlmfat2013n1}.

The Hill--Schr\"{o}dinger operator $\mathrm{S}(q)$ with strongly singular potential $q$ is lower semibounded and 
self-adjoint, its spectrum is absolutely continuous and has a band and gap structure as in the classical case 
\cite{HrMk2001, Kr2003, MiMl6, DjMt4, MiSb}. 
The endpoints of spectral gaps satisfy the inequalities \eqref{InEq}. 
For even/odd numbers $n\in \mathbb{Z}_{+}$ they are eigenvalues of the periodic/semiperiodic problems 
on the interval $[0,1]$ \cite[Theorem~C]{MiMl6}.

We say that the weighted sequence $\omega=\{\omega(k)\}_{k\in \mathbb{Z}}$ belongs to $\mathrm{I}_{-1}$, 
if it satisfies the following conditions:
\begin{align*}\label{eq_50}
 \mathtt{(i)}\hspace{5pt} & \omega(k)=(1+|k|)^{-1}, & & \quad s=1, \\
 \mathtt{(ii)}\hspace{5pt} & |k|^{s}\ll \omega(k)\ll |k|^{1+2s-\delta}\quad \forall\delta>0, & & \quad s\in (-1,0), \\
 \mathtt{(iii)}\hspace{5pt} & |k|^{s}\ll \omega(k)\ll |k|^{1+s}, & & \quad s\in [0,\infty). \hspace{140pt}
\end{align*}

We say that the weighted sequence $\omega=\{\omega(k)\}_{k\in \mathbb{Z}}$ belongs to $\mathrm{M}_{-1}$, 
if it can be represented as:
\begin{equation*}\label{eq_52}
 \omega(k)=\dfrac{\omega^{\ast}(k)}{1+|k|},\quad k\in \mathbb{Z},\qquad \omega^{\ast}=\{\omega^{\ast}(k)\}_{k\in 
\mathbb{Z}}\in \mathrm{M}_{0}. \hspace{150pt}
\end{equation*}

Suppose that a weighted sequence $\omega=\{\omega(k)\}_{k\in \mathbb{Z}}$ belongs either to the class 
$\mathrm{I}_{-1}$, or to the class $\mathrm{M}_{-1}$, then 
\begin{equation}\label{eq_54}
 q\in H^{\omega}(\mathbb{T})\Leftrightarrow \{\gamma_{q}(n)\}_{n\in \mathbb{N}}\in  h^{\omega}(\mathbb{N}).
\end{equation} 
The statement \eqref{eq_54} for the case $\omega\in \mathrm{I}_{-1}$ is proved below 
(in a weaker form, this assertion was proved earlier by the authors \cite{MiMlmfat2011n3}), 
also for the case $\omega\in\mathrm{M}_{-1}$ the statement\eqref{eq_54} 
was proved in~\cite{DjMt3}. 
Note that $\mathrm{I}_{0}$ and $\mathrm{M}_{0}$, as well as $\mathrm{I}_{-1}$ and $\mathrm{M}_{-1}$,
intersect, but do not cover each other.

Let us consider the map $\gamma: q\mapsto\{\gamma_{q}(n)\}_{n\in \mathbb{N}}$. 
Then, according to Korotyaev \cite[Theorem~1.1]{Kr2003}, map $\gamma$ maps $H^{-1}(\mathbb{T})$ onto 
$h_{+}^{-1}(\mathbb{N})$,
\begin{equation}\label{eq_56}
 \gamma(H^{-1}(\mathbb{T}))=h_{+}^{-1}(\mathbb{N}).
\end{equation} 

\begin{theorem}\label{th_12}
Suppose that $q\in H^{-1}(\mathbb{T})$ and that either $\omega\in\mathrm{I}_{-1}$ or $\omega\in \mathrm{M}_{-1}$. 
Then the map 
\begin{equation*}
  \gamma: H^{-1}(\mathbb{T})\ni q\mapsto \{\gamma_{q}(n)\}_{n\in \mathbb{N}} \in h_{+}^{-1}(\mathbb{N})
\end{equation*}
satisfies following equalities:
\begin{align*}
 \mathtt{(i)}\quad &\gamma\left(H^{\omega}(\mathbb{T})\right) = h_{+}^{\omega}(\mathbb{N}), \hspace{25pt}\\
 \mathtt{(ii)}\quad &\gamma^{-1}\left(h_{+}^{\omega}(\mathbb{N})\right) = H^{\omega}(\mathbb{T}). \hspace{110pt}
\end{align*}
\end{theorem}

%%%%%%%%%%%%%%%%%%%%%%%%%%%%%%%%%%%%%%%%%%%%%%%%%%%%%%%%%%%%%%%%%%%%%%%%%%%%%%%%%%%%%%%%%%%%%%%%%%%%%%%%%%%%%%%%%%%%%%%%%%%%
%%%%%%%%%%%%%%%%%%%%%%%%%%%%%%%%%%%%%%%%%%%%%%%%%%%%%%%%%%%%%%%%%%%%%%%%%%%%%%%%%%%%%%%%%%%%%%%%%%%%%%%%%%%%%%%%%%%%%%%%%%%%
% Theorem 3
%%%%%%%%%%%%%%%%%%%%%%%%%%%%%%%%%%%%%%%%%%%%%%%%%%%%%%%%%%%%%%%%%%%%%%%%%%%%%%%%%%%%%%%%%%%%%%%%%%%%%%%%%%%%%%%%%%%%%%%%%%%%
%%%%%%%%%%%%%%%%%%%%%%%%%%%%%%%%%%%%%%%%%%%%%%%%%%%%%%%%%%%%%%%%%%%%%%%%%%%%%%%%%%%%%%%%%%%%%%%%%%%%%%%%%%%%%%%%%%%%%%%%%%%%
\section{The Proofs}
\begin{proof}[Proof of Theorem~\ref{th_10}]
Due to Garnett and Trubowitz~\cite{GrTr1984, GrTr1987} it occurs that for any sequence $\{\gamma(n)\}_{n\in \mathbb{N}}\in 
h_{+}^{0}(\mathbb{N})$ we can place the open intervals $I_{n}$ of the lengths $\gamma(n)$ (to the length 0 corresponds 
point) on the positive semi-axis $(0,\infty)$ in a such single way that there exists a potential $q\in L^{2}(\mathbb{T})$ for 
which the sequence $\{\gamma(n)\}_{n\in \mathbb{N}}$ is a sequence of the lengths of spectral gaps of the 
Hill--Schr\"{o}dinger operator $S(q)$, i.e., the map $\gamma$ maps the space $L^{2}(\mathbb{T})$ \textit{onto} the sequence 
space $h_{+}^{0}(\mathbb{N})$:
\begin{equation}\label{MpEq}
  \gamma(L^{2}(\mathbb{T}))=h_{+}^{0}(\mathbb{N}).
\end{equation}
And, as a consequence, we also have
\begin{equation}\label{MpEqIn}
  \gamma^{-1}\left(h_{+}^{0}(\mathbb{N})\right)=L^{2}(\mathbb{T}).
\end{equation}

The case $\omega\in\mathrm{I}_{0}$ was investigated by the authors in \cite{MiMlOTAA2012}. 

Let $\omega\in \mathrm{M}_{0}$. From statement \eqref{eq_30} we get
\begin{equation}\label{eq_44}
  \gamma\left(H^{\omega}(\mathbb{T})\right) \subset h_{+}^{\omega}(\mathbb{N}).
\end{equation}
To establish the equality (i) of Theorem \ref{th_10} it is necessary to prove the inverse inclusion 
in formula \eqref{eq_44}. 
So, let $\{\gamma(n)\}_{n\in \mathbb{N}}$ be an arbitrary sequence in the space $h_{+}^{\omega}(\mathbb{N})$. 
Then $\{\gamma(n)\}_{n\in \mathbb{N}}\in h_{+}^{0}(\mathbb{N})$. 
Due to \eqref{MpEq} a potential $q\in L^{2}(\mathbb{T})$ exists, 
such that the sequence $\{\gamma(n)\}_{n\in \mathbb{N}}\in h_{+}^{0}(\mathbb{N})$ 
is its sequence of the lengths of spectral gaps. 
Since by assumption $\{\gamma(n)\}_{n\in \mathbb{N}}\in h_{+}^{\omega}(\mathbb{N})$ due to \eqref{eq_30} 
we conclude that $q\in H^{\omega}(\mathbb{T})$ and as a consequence 
$\{\gamma(n)\}_{n\in \mathbb{N}}\in \gamma\left(H^{\omega}(\mathbb{T})\right)$. 
Therefore the inclusion
\begin{equation}\label{eq_46}
  \gamma\left(H^{\omega}(\mathbb{T})\right) \supset h_{+}^{\omega}(\mathbb{N})
\end{equation}
holds.

Inclusions \eqref{eq_44} and \eqref{eq_46} give the equality (i).

Now, let us prove the equality (ii) of Theorem \ref{th_10}. Let $q$ be an arbitrary function in the space 
$H^{\omega}(\mathbb{T})$. Then, due to statement \eqref{eq_30}, we have
$\gamma_{q}=\{\gamma_{q}(n)\}_{n\in \mathbb{N}}\in h_{+}^{\omega}(\mathbb{N})$ 
and as a consequence $q\in \gamma^{-1}\left(h_{+}^{\omega}(\mathbb{N})\right)$. 
Therefore
\begin{equation}\label{eq_48}
  \gamma^{-1}\left(h_{+}^{\omega}(\mathbb{N})\right) \supset H^{\omega}(\mathbb{T}).
\end{equation}

Conversely, let $\{\gamma(n)\}_{n\in \mathbb{N}}$ be an arbitrary sequence from the space $h_{+}^{\omega}(\mathbb{N})$. Then 
due to \eqref{MpEqIn} we have $\gamma^{-1}\left(\{\gamma(n)\}_{n\in \mathbb{N}}\right)\subset L^{2}(\mathbb{T})$. Taking 
into account \eqref{eq_30} we conclude that $\gamma^{-1}\left(\{\gamma(n)\}_{n\in \mathbb{N}}\right)\subset 
H^{\omega}(\mathbb{T})$, that is 
\begin{equation}\label{eq_50}
  \gamma^{-1}\left(h_{+}^{\omega}(\mathbb{N})\right) \subset H^{\omega}(\mathbb{T}).
\end{equation}

Inclusions \eqref{eq_48} and \eqref{eq_50} give the equality (ii) of Theorem \ref{th_10}.

The proof of Theorem~\ref{th_10} is complete.
\end{proof}

\begin{proof}[Proof of Formula~\eqref{eq_54}]
Notice that for the case $\omega(k)=(1+|k|)^{s}$, $s\in [-1,\infty)$, 
the relation~\eqref{eq_54} has the form
\begin{equation}\label{eq_54.10}
  q\in H^{s}(\mathbb{T})\Leftrightarrow \{\gamma_{q}(n)\}_{n\in \mathbb{N}}\in  h^{s}(\mathbb{N}), \qquad s\in [-1,\infty).
\end{equation}

The limiting case $s=-1$ was treated by Korotyaev \cite{Kr2003}. 
The proof of statement~\eqref{eq_54.10} was completed in \cite{DjMt3}. 
Earlier \eqref{eq_54.10} was established by the authors \cite{MiMlmfat2009n1} 
under a stronger assumption $q\in H^{-1+}(\mathbb{T})$ and $s>-1$.

Furthermore, if $q\in H^{s}(\mathbb{T})$, $s\in [-1,\infty)$, 
then for the lengths of spectral gaps the following asymptotic formula hold \cite{Kr2003, MiMlmfat2009n1}:
\begin{align}
  \gamma_{q}(n) & =2|\widehat{q}(n)|+h^{-1}(n) & & \text{if}\quad 
  s=-1,\label{eq_54.12.0} \\
  \gamma_{q}(n) & =2|\widehat{q}(n)|+h^{1+2s-\delta}(n)\quad \forall\delta>0 & & \text{if}\quad 
  s\in(-1,0),\label{eq_54.12.1} \\
  \gamma_{q}(n) & =2|\widehat{q}(n)|+h^{1+s}(n) & & \text{if}\quad s\in[0,\infty).\hspace{100pt}\mbox{ }\label{eq_54.12.2}
\end{align}
Let us also recall that if $\omega_{1}\gg \omega_{2}$, i. e., $\omega_{1}(k)\gg \omega_{2}(k)$, 
$k\in \mathbb{Z}$, then
\begin{equation}\label{eq_54.14}
 H^{\omega_{1}}(\mathbb{T})\hookrightarrow H^{\omega_{2}}(\mathbb{T}),\qquad 
 h^{\omega_{1}}(\mathbb{N})\hookrightarrow h^{\omega_{2}}(\mathbb{N}).
\end{equation} 

Let $q\in H^{\omega}(\mathbb{T})$ and $\omega\in \mathrm{I}_{-1}$, 
then taking into account \eqref{eq_54.14} we have $q\in H^{s}(\mathbb{T})$, $s\in [-1,\infty)$, 
as $\omega(k)\gg |k|^{s}$. 
Taking into account that $\omega \in \mathrm{I}_{-1}$, together with \eqref{eq_54.14}, 
from \eqref{eq_54.12.0} -- \eqref{eq_54.12.2} we get:
\begin{equation*}
 \gamma_{q}(n)=2|\widehat{q}(n)|+h^{\omega}(n),
\end{equation*}
i. e., $\{\gamma_{q}(n)\}_{n\in \mathbb{N}}\in  h^{\omega}(\mathbb{N})$.

Now, let $\{\gamma_{q}(n)\}_{n\in \mathbb{N}}\in  h^{\omega}(\mathbb{N})$, then due to \eqref{eq_54.14} we have
$\{\gamma_{q}(n)\}_{n\in \mathbb{N}}\in  h^{s}(\mathbb{N})$, $s\in [-1,\infty)$, and, as consequence, from \eqref{eq_54.10} 
we get $q\in H^{s}(\mathbb{T})$, $s\in [-1,\infty)$, and the asymptotics \eqref{eq_54.12.0} -- \eqref{eq_54.12.2} hold. 
Taking into account that $\omega \in  \mathrm{I}_{-1}$ and \eqref{eq_54.14} we have:
\begin{equation*}
 \gamma_{q}(n)=2|\widehat{q}(n)|+h^{\omega}(n),
\end{equation*}
from where we get necessary result $q\in H^{\omega}(\mathbb{T})$.

The statement~\eqref{eq_54} for the case $\omega \in  \mathrm{I}_{-1}$ is completely proved.
\end{proof}

\begin{proof}[Proof of Theorem~\ref{th_12}]
From statement \eqref{eq_54} we get
\begin{equation}\label{eq_58}
  \gamma\left(H^{\omega}(\mathbb{T})\right) \subset h_{+}^{\omega}(\mathbb{N}).
\end{equation}
To establish the equality (i) of Theorem \ref{th_12} it is necessary to prove the inverse inclusion in formula 
\eqref{eq_58}. 
So, let $\{\gamma(n)\}_{n\in \mathbb{N}}$ be an arbitrary sequence in the space $h_{+}^{\omega}(\mathbb{N})$. Then 
$\{\gamma(n)\}_{n\in \mathbb{N}}\in h_{+}^{-1}(\mathbb{N})$. Due to \eqref{eq_56} a
potential $q\in H^{-1}(\mathbb{T})$ exists, such that the sequence $\{\gamma(n)\}_{n\in \mathbb{N}}\in 
h_{+}^{-1}(\mathbb{N})$ is its sequence of the lengths of spectral gaps. Since by assumption $\{\gamma(n)\}_{n\in 
\mathbb{N}}\in h_{+}^{\omega}(\mathbb{N})$ due to \eqref{eq_54} we conclude that $q\in H^{\omega}(\mathbb{T})$ and as 
consequence $\{\gamma(n)\}_{n\in \mathbb{N}}\in \gamma\left(H^{\omega}(\mathbb{T})\right)$. Therefore the inclusion
\begin{equation}\label{eq_60}
  \gamma\left(H^{\omega}(\mathbb{T})\right) \supset h_{+}^{\omega}(\mathbb{N})
\end{equation}
holds.

Inclusions \eqref{eq_58} and \eqref{eq_60} give the equality (i).

Now, let us prove the equality (ii) of Theorem \ref{th_12}. Let $q$ be an arbitrary function in the space 
$H^{\omega}(\mathbb{T})$. Then, due to statement 
\eqref{eq_54}, we have $\gamma_{q}=\{\gamma_{q}(n)\}_{n\in \mathbb{N}}\in h_{+}^{\omega}(\mathbb{N})$, i. e., $q\in 
\gamma^{-1}\left(h_{+}^{\omega}(\mathbb{N})\right)$. Therefore
\begin{equation}\label{eq_62}
  \gamma^{-1}\left(h_{+}^{\omega}(\mathbb{N})\right) \supset H^{\omega}(\mathbb{T}).
\end{equation}

Conversely, let $\{\gamma(n)\}_{n\in \mathbb{N}}$ be an arbitrary sequence from the space $h_{+}^{\omega}(\mathbb{N})$. Then 
due to \eqref{eq_56} we have $\gamma^{-1}\left(h_{+}^{-1}(\mathbb{N})\right)=H^{-1}(\mathbb{T})$, and therefore 
$\gamma^{-1}\left(\{\gamma(n)\}_{n\in \mathbb{N}}\right)\subset H^{-1}(\mathbb{T})$. Taking into account \eqref{eq_54} we 
conclude that $\gamma^{-1}\left(\{\gamma(n)\}_{n\in \mathbb{N}}\right)\subset H^{\omega}(\mathbb{T})$, that is
\begin{equation}\label{eq_64}
  \gamma^{-1}\left(h_{+}^{\omega}(\mathbb{N})\right) \subset H^{\omega}(\mathbb{T}).
\end{equation}

Inclusions \eqref{eq_62} and \eqref{eq_64} give the equality (ii) of Theorem \ref{th_12}.

The proof of Theorem \ref{th_12} is complete. 
\end{proof}

%%%%%%%%%%%%%%%%%%%%%%%%%%%%%%%%%%%%%%%%%%%%%%%%%%%%%%%%%%%%%%%%%%%%%%%%%%%%%%%%%%%%%%%%%%%%%%%%%%%%%%%%%%%%%%%%%%%%%%%%%%%%
%%%%%%%%%%%%%%%%%%%%%%%%%%%%%%%%%%%%%%%%%%%%%%%%%%%%%%%%%%%%%%%%%%%%%%%%%%%%%%%%%%%%%%%%%%%%%%%%%%%%%%%%%%%%%%%%%%%%%%%%%%%%
\vspace{25pt}

\textit{Acknowledgments.} This work was partially supported by Project No.~03-01-12/2 of National Academy of Science of 
Ukraine.

%%%%%%%%%%%%%%%%%%%%%%%%%%%%%%%%%%%%%%%%%%%%%%%%%%%%%%%%%%%%%%%%%%%%%%%%%%%%%%%%%%%%%%%%%%%%%%%%%%%%%%%%%%%%%%%%%%%%%%%%%%%%
%%%%%%%%%%%%%%%%%%%%%%%%%%%%%%%%%%%%%%%%%%%%%%%%%%%%%%%%%%%%%%%%%%%%%%%%%%%%%%%%%%%%%%%%%%%%%%%%%%%%%%%%%%%%%%%%%%%%%%%%%%%%
%\newpage
%%%%%%%%%%%%%%%%%%%%%%%%%%%%%%%%%%%%%%%%%%%%%%%%%%%%%%%%%%%%%%%%%%%%%%%%%%%%%%%%%%%%%%%%%%%%%%%%%%%%%%%%%%%%%%%%%%%%%%%%%%%%

\end{document}